\documentclass[11pt,a4paper]{amsart}
\usepackage[T1]{fontenc}
\usepackage[utf8]{inputenc}
\usepackage{lmodern}
\usepackage[a4paper,margin=27mm]{geometry}
\usepackage{amsmath,amssymb,amsthm,mathtools,booktabs}
\usepackage{microtype}
\usepackage{enumitem}
\usepackage{listings}
\usepackage[hidelinks]{hyperref}
\hypersetup{pdftitle={Hyponormality of generalized Cesaro matrices of every positive integer order},pdfauthor={Yicen Ma},pdfsubject={Computer-assisted proof with exact rational certificates}}
\newtheorem{theorem}{Theorem}[section]
\newtheorem{lemma}[theorem]{Lemma}
\newtheorem{proposition}[theorem]{Proposition}

\theoremstyle{remark}
\newtheorem{remark}[theorem]{Remark}
\numberwithin{equation}{section}
\DeclareMathOperator{\diag}{diag}

\newcommand{\NN}{\mathbb N}
\newcommand{\HH}{\ell^2(\mathbb N_0)}
\newcommand{\E}{\mathbb E}
\newcommand{\Beta}{\mathrm B}
\newcommand{\fall}[2]{#1^{\underline{#2}}}
\newcommand{\cS}{\mathcal S}
\newcommand{\odd}{\mathrm{odd}}
\newcommand{\dd}{\,\mathrm d}

\newcommand{\tht}{\Theta_{\alpha,m}}
\setlist[enumerate]{leftmargin=*,itemsep=2pt,topsep=4pt}
\allowdisplaybreaks[1]
\title[Generalized Ces\`aro matrices]{Hyponormality of generalized Ces\`aro matrices\protect\\ of every positive integer order}
\author{Yicen Ma}
\address{Zhejiang University, Hangzhou, Zhejiang, China}
\date{September 18, 2026}
\thanks{The finite verification programs and exact certificate data form an electronic supplement to this manuscript.}

\begin{document}
\begin{abstract}
For a positive integer $m$ and a real parameter $\alpha>-1$, consider the generalized Ces\`aro matrix on $\ell^2(\mathbb N_0)$ with entries
$m(n-j+1)_{m-1}/(n+\alpha+1)_m$ for $j\le n$.
We give a computer-assisted proof that this operator is hyponormal if and only if $\alpha\ge0$.
The main algebraic ingredient is an explicit decomposition of an auxiliary defect operator into a positive diagonal operator and at most $m$ rank-one operators. Discrete Gram polynomials determine the signs and coefficients of these terms. For nonnegative parameters, weighted estimates reduce positivity to scalar inequalities. A uniform analytic estimate covers every $m\ge32768$, and finite rational certificates cover the remaining orders and entire parameter intervals, without parameter sampling or truncating an infinite operator. The proof also gives a uniform weighted lower bound for the self-commutator. For $-1<\alpha<0$, finite-support vectors give negative quadratic forms for every order.
\end{abstract}
\maketitle

\section{Introduction and main result}
Write $(z)_q=z(z+1)\cdots(z+q-1)$, with $(z)_0=1$. For $m\in\NN=\{1,2,\ldots\}$ and $\alpha>-1$, let $M_{\alpha,m}$ denote the lower triangular matrix
\begin{equation}\label{eq:matrix}
 (M_{\alpha,m})_{nj}=
 \begin{cases}
 \displaystyle\frac{m(n-j+1)_{m-1}}{(n+\alpha+1)_m},&0\le j\le n,\\[4pt]
 0,&j>n.
 \end{cases}
\end{equation}
This is the positive-integer-order case of the generalized Ces\`aro matrices $(C^{(\alpha)},m)$ considered by Rhaly \cite{R18,R19}. A bounded operator $T$ on a Hilbert space is hyponormal when $T^*T-TT^*\ge0$.

For $m=2$, Rhaly proved hyponormality for $\alpha\in\{0\}\cup[1,\infty)$ and asked whether the interval $0<\alpha<1$ could be included \cite[Conjecture 3.2]{R18}. The corresponding third-order result covers $\alpha\in\{0,1\}\cup[2,\infty)$ \cite[Corollary 2.4]{R19}. The closing remark of \cite{R19} also invites consideration of arbitrary positive integer orders.

The unshifted case $\alpha=0$ must be distinguished from this parameter problem. Its hyponormality already follows from the classical subnormality of bounded Hausdorff operators, as explained in the update to \cite{R17}. Le and Rhaly obtained the diagonal interrupter formula for Ces\`aro matrices of all orders in \cite{LR18}. On the other hand, one cannot simply assume subnormality for noninteger shifts: already at order one, the distinction between integer and noninteger parameters is substantive; see \cite{RT19}. The result below concerns hyponormality and does not assert subnormality.

Set
\begin{equation}\label{eq:HD}
 H_m(t)=(t+1)_m,\qquad D_m(t)=(t+m+1)_m,
\end{equation}
and define
\begin{equation}\label{eq:R}
 R_{\alpha,m}=\diag\bigl(r_n(\alpha)\bigr)_{n\ge0},\qquad
 r_n(\alpha)=1-\frac{H_m(n+\alpha)}{D_m(n+\alpha)}.
\end{equation}
All the diagonal entries are positive for $\alpha>-1$.

\begin{theorem}\label{thm:main}
For every positive integer $m$ and every real $\alpha>-1$, the matrix $M_{\alpha,m}$ defines a bounded operator on $\HH$, and
\[
 M_{\alpha,m}\text{ is hyponormal}\quad\Longleftrightarrow\quad\alpha\ge0.
\]
For every $\alpha\ge0$, one has the quantitative estimate
\begin{equation}\label{eq:mainweight}
 M_{\alpha,m}^*M_{\alpha,m}-M_{\alpha,m}M_{\alpha,m}^*
 \ge\frac1{36}M_{\alpha,m}^*R_{\alpha,m}M_{\alpha,m}.
\end{equation}
\end{theorem}

The diagonal entries of $R_{\alpha,m}$ tend to zero. Thus \eqref{eq:mainweight} is a weighted lower bound, not a positive spectral gap relative to the identity. Neither the constant $1/36$ nor the order threshold used below is asserted to be optimal.

The proof has two distinct parts. The operator identities, parameter reductions and all infinite-sum estimates are proved analytically. The remaining finite inequalities are checked with integer and rational arithmetic. Specifically, orders $1$ through $20$ are treated by polynomial coefficient certificates, orders $21$ through $32767$ by directed-rounding bounds on complete order--parameter rectangles, and all larger orders by an explicit analytic estimate. The supplement recomputes these certificates; the finite computation is a part of the proof, not numerical evidence for an unproved limiting assertion.

Throughout, $c_{00}$ denotes the finitely supported sequences. For a vector $v\in\HH$, $vv^*$ is the rank-one operator whose matrix entries are $v_i\overline{v_j}$. All vectors used in rank-one formulas below are real. Standard beta and gamma identities used in the proof can be found in \cite{DLMF}.

\section{Finite differences and an explicit defect decomposition}\label{sec:defect}
We first establish a formula valid for every $m\ge1$ and $\alpha>-1$. This section does not use the finite certificates.

\subsection{Boundedness and a finite-support right inverse}
Fix $m,\alpha$ and put $c=\min\{1,1+\alpha\}>0$. For $j\le n$,
\[
 0\le (M_{\alpha,m})_{nj}\le \frac{m^m}{c^m(n+1)}.
\]
The Schur test with weights $(n+1)^{-1/2}$ proves boundedness, using
\[
 \sum_{j=0}^n(j+1)^{-1/2}\le2\sqrt{n+1},\qquad
 \sum_{n=j}^{\infty}(n+1)^{-3/2}\le\frac3{\sqrt{j+1}}.
\]
Define $L=L_{\alpha,m}$ only on $c_{00}$ by
\begin{equation}\label{eq:L}
 Le_j=\frac{(j+\alpha+1)_m}{m!}
       \sum_{s=0}^m(-1)^s\binom ms e_{j+s}.
\end{equation}
The coefficient identity $(1-z)^m(1-z)^{-m}=1$ gives
\begin{equation}\label{eq:ML}
 M_{\alpha,m}Lx=x\qquad(x\in c_{00}).
\end{equation}
Moreover, $L(c_{00})$ is dense in $\HH$. Indeed, a vector orthogonal to every $Le_j$ has vanishing $m$th forward differences and is therefore a polynomial sequence of degree at most $m-1$. The only such sequence in $\ell^2$ is zero. No boundedness of $L$ is used or claimed.

Initially define $B_0=M_{\alpha,m}^*L$ on $c_{00}$. Each column of $B_0$ has finite support. We will prove its boundedness from its quadratic form, before using any identity involving its bounded extension.

\subsection{An auxiliary polynomial and the columns of \texorpdfstring{$B_0$}{B0}}
Write $\fall{x}{q}=x(x-1)\cdots(x-q+1)$ and set
\begin{equation}\label{eq:Fpoly}
 F_m(x,z)=\sum_{q=0}^{m-1}(-1)^q
 \frac{(m+q)!}{(q!)^2(m-q-1)!}
 \fall{x}{q}(z+1)_{m-q-1}.
\end{equation}
For this section, abbreviate $d_j=D_m(j+\alpha)$.

\begin{lemma}\label{lem:columns}
The $j$th column of $B_0$ is
\begin{equation}\label{eq:Bentries}
 b_{ij}=\begin{cases}
 F_m(i+\alpha,j+\alpha)/d_j,&0\le i\le j+m-1,\\
 (-1)^m H_m(j+\alpha)/d_j,&i=j+m,\\
 0,&i>j+m.
 \end{cases}
\end{equation}
\end{lemma}
\begin{proof}
For $i\le j$, insert \eqref{eq:L} into $M^*Le_j$ and use
\[
 (z-x+1)_{m-1}=\sum_{q=0}^{m-1}(-1)^q\binom{m-1}{q}
 \fall{x}{q}(z+1)_{m-1-q}.
\]
The finite difference identity
\begin{equation}\label{eq:betadifference}
 \sum_{s=0}^m(-1)^s\binom ms\frac1{(z+s)_{q+1}}
 =\frac{(q+1)_m}{(z)_{q+m+1}}
\end{equation}
follows by writing the reciprocal rising factorial as a beta integral and summing $(1-t)^m$ under that finite integral. Apply \eqref{eq:betadifference} with $z=j+\alpha+m-q$. Cancellation gives exactly \eqref{eq:Fpoly} divided by $d_j$.

For $i=j+t$, $1\le t<m$, the terms formally added below the diagonal contain a zero factor in a rising factorial of length $m-1$; the same formula is consequently valid. At $i=j+m$, the actual column has only its final finite-difference term. The formal polynomial formula has one additional term, from $s=0$, equal to $(-1)^{m-1}$. Hence the actual final entry exceeds the polynomial value by $(-1)^m$, and is the value displayed in \eqref{eq:Bentries}. All later entries vanish.
\end{proof}

\subsection{The unshifted Gram matrix and a boundary sum}
At $\alpha=0$, for each $0\le q<m$, finite convolution gives
\begin{equation}\label{eq:polymoment}
 M_{0,m}\bigl(\fall{j}{q}\bigr)(n)
 =\binom{m+q}{q}^{-1}\fall{n}{q}.
\end{equation}
This is a rowwise finite algebraic identity; the polynomial sequences are not being regarded as elements of $\ell^2$. Since the $m$th difference of the right-hand side vanishes, \eqref{eq:polymoment} implies
\[
 \sum_i\fall{i}{q}b_{ij}=0\qquad(0\le q<m,\ \alpha=0).
\]
For $i<j$, the $j$th column is a polynomial of degree at most $m-1$ on the entire support of the $i$th column. Thus distinct columns are orthogonal. If $g_j$ is the final entry of the $j$th column, the final-entry correction in Lemma~\ref{lem:columns} also gives
\begin{equation}\label{eq:norm0}
 \|B_0e_j\|^2=(-1)^mg_j=\frac{H_m(j)}{D_m(j)}
 \qquad(\alpha=0).
\end{equation}
This recovers the known unshifted diagonal interrupter identity; compare \cite{LR18}.

For an integer $A\ge0$, the matrix $B_{A,m}$ is the tail of the matrix $B_{0,m}$ after deleting the first $A$ rows and columns. Subtracting the omitted rows from \eqref{eq:norm0} therefore yields
\begin{equation}\label{eq:boundaryinteger}
 (I-B_{A,m}^*B_{A,m})_{ij}
 =r_i(A)\delta_{ij}
 +\frac{\displaystyle\sum_{x=0}^{A-1}F_m(x,i+A)F_m(x,j+A)}
 {D_m(i+A)D_m(j+A)}.
\end{equation}
At this point this identity can be read entrywise, involving only finite column sums.

For a polynomial $f(x)$, let $\cS_\alpha[f]$ denote the polynomial continuation in $A$ of $\sum_{x=0}^{A-1}f(x)$, evaluated at $A=\alpha$. If the coefficients of $f$ also involve $\alpha$, the continuation is performed termwise in $x$. After clearing denominators, both sides of \eqref{eq:boundaryinteger} are polynomials in $A$. Their equality at all nonnegative integers proves, for every real $\alpha>-1$,
\begin{equation}\label{eq:boundarygeneral}
 (I-B_0^*B_0)_{ij}
 =r_i(\alpha)\delta_{ij}
 +\frac{\cS_\alpha[F_m(x,i+\alpha)F_m(x,j+\alpha)]}{d_i d_j}.
\end{equation}
The left side here denotes the finite-column Gram form on $c_{00}$. Importantly, $\cS_\alpha$ need not be positive when $\alpha$ is not an integer. We continue algebraic identities, not positivity of an inner product.

\subsection{Discrete Gram orthogonalization}
Define the monic discrete Gram polynomials
\begin{equation}\label{eq:Qr}
 Q_r(x;\alpha)=\frac{r!}{(2r)!}\Delta^r
 \bigl[\fall{x}{r}\fall{(x-\alpha)}{r}\bigr],
 \qquad \Delta f(x)=f(x+1)-f(x).
\end{equation}
\begin{lemma}\label{lem:gram}
The polynomials $Q_r$ are orthogonal for the formal functional $\cS_\alpha$, and
\begin{equation}\label{eq:gramnorm}
 \cS_\alpha[Q_r^2]
 =\frac{(r!)^4}{(2r)!}\binom{\alpha+r}{2r+1}
 =\frac{\alpha\prod_{j=1}^r(\alpha^2-j^2)}
 {(2r+1)\binom{2r}{r}^2}.
\end{equation}
For $q\ge r$,
\begin{equation}\label{eq:mixedmoment}
 \cS_\alpha[\fall{x}{q}Q_r(x;\alpha)]
 =\frac{(r!)^2(q!)^2}{(2r)!(q-r)!}
   \binom{\alpha+r}{q+r+1};
\end{equation}
for $q<r$, the mixed moment is zero.
\end{lemma}
\begin{proof}
First take an integer $A$ larger than all degrees involved. The polynomial
$W(x)=\fall{x}{r}\fall{(x-A)}{r}$ vanishes at $0,\ldots,r-1$ and $A,\ldots,A+r-1$. Repeated finite summation by parts therefore gives
\[
 \sum_{x=0}^{A-1}\Delta^rW(x)f(x)
 =(-1)^r\sum_{x=r}^{A-1}W(x)\nabla^rf(x),
 \qquad\nabla f(x)=f(x)-f(x-1).
\]
This proves orthogonality to lower degrees. For $f(x)=\fall{x}{q}$, use
$\nabla^r f(x)=q!\fall{(x-r)}{q-r}/(q-r)!$.
The remaining sum, with its signs cancelled, is a constant multiple of
\[
 \sum_{x=q}^{A-1}\binom{x}{q}\binom{A-x+r-1}{r}
 =\binom{A+r}{q+r+1},
\]
by Vandermonde's identity. Including the normalization in \eqref{eq:Qr} gives \eqref{eq:mixedmoment}. Since $Q_r$ is monic, its squared norm equals its mixed moment with $\fall{x}{r}$, giving \eqref{eq:gramnorm}. All these expressions are polynomial in $A$, so the identities extend to every real $\alpha$.
\end{proof}

For $0\le r<m$, put $d=m-r-1$ and define the monic degree-$d$ polynomial
\begin{equation}\label{eq:pold}
 p_{m,r}(n;\alpha)=\sum_{s=0}^{d}\binom ds
 \frac{(m+r+1)_s(r+1-\alpha)_s}{(2r+2)_s}
 (n+\alpha+1)_{d-s}.
\end{equation}
Equations \eqref{eq:Fpoly} and \eqref{eq:mixedmoment} imply
\begin{equation}\label{eq:Fexpansion}
 F_m(x,n+\alpha)=\sum_{r=0}^{m-1}(-1)^r
 \frac{(m+r)!}{(r!)^2(m-r-1)!}\,
 p_{m,r}(n;\alpha)Q_r(x;\alpha).
\end{equation}
For clarity, the coefficient calculation first takes $\alpha=A$ a sufficiently large integer and divides the mixed moment by \eqref{eq:gramnorm}. In the term $q=r+s$, the ratio of binomial coefficients becomes
\[
 \frac{\binom{A+r}{2r+s+1}}{\binom{A+r}{2r+1}}
 =\frac{(-1)^s(r+1-A)_s}{(2r+2)_s}.
\]
This cancels the additional sign $(-1)^s$ and gives \eqref{eq:pold}. The resulting expression has no parameter-dependent denominator, so \eqref{eq:Fexpansion} extends polynomially to all $\alpha$, including parameters at which some formal squared norms vanish.

\begin{theorem}[Defect decomposition]\label{thm:defect}
For $m\ge1$ and $\alpha>-1$, $B_0=M_{\alpha,m}^*L_{\alpha,m}$ has a unique bounded extension $B=B_{\alpha,m}$, and
\begin{align}
 BM_{\alpha,m}&=M_{\alpha,m}^*,\label{eq:BM}\\
 I-B^*B&=R_{\alpha,m}+\sum_{r=0}^{m-1}
 \lambda_{m,r}(\alpha)v_{m,r}v_{m,r}^*,\label{eq:defect}\\
 M_{\alpha,m}^*M_{\alpha,m}-M_{\alpha,m}M_{\alpha,m}^*
 &=M_{\alpha,m}^*(I-B^*B)M_{\alpha,m},\label{eq:transfer}
\end{align}
where
\begin{equation}\label{eq:vlambda}
 v_{m,r}(n)=\frac{p_{m,r}(n;\alpha)}{D_m(n+\alpha)},\qquad
 \lambda_{m,r}(\alpha)
 =(2r+1)\binom{m+r}{2r+1}^{\!2}
 \alpha\prod_{j=1}^r(\alpha^2-j^2).
\end{equation}
\end{theorem}
\begin{proof}
Substitute \eqref{eq:Fexpansion} into \eqref{eq:boundarygeneral} and apply Lemma~\ref{lem:gram}. Simplification of the squared expansion coefficients times the norms gives precisely \eqref{eq:vlambda}.

First interpret \eqref{eq:defect} on $c_{00}$. Since $p_{m,r}$ is monic of degree $m-r-1$, $v_{m,r}(n)=O(n^{-r-1})$, so all these vectors belong to $\ell^2$. The right side of \eqref{eq:defect} is a bounded quadratic form. Consequently $\|B_0x\|^2\le C\|x\|^2$ on $c_{00}$, and $B_0$ extends boundedly. The form identity then holds as a bounded-operator identity. For $x\in c_{00}$,
\[
 BM(Lx)=Bx=M^*Lx.
\]
The density of $L(c_{00})$ proves \eqref{eq:BM}. Taking adjoints and multiplying gives \eqref{eq:transfer}.
\end{proof}

\section{Parameter reduction and the negative directions}\label{sec:reduction}
\subsection{A positive-coefficient formula}
\begin{lemma}\label{lem:pnew}
With $d=m-r-1$, one has
\begin{equation}\label{eq:pnew}
 p_{m,r}(n;\alpha)=\sum_{q=0}^{d}\binom dq^{\!2}q!
 \frac{(r+1-\alpha)_q}{(2r+2)_q}
 (n+r+2+q)_{d-q}.
\end{equation}
For $r\ge1$, $n\ge0$ and $0\le a\le b\le1$,
\begin{equation}\label{eq:pmono}
 0<p_{m,r}(n;b)\le p_{m,r}(n;a).
\end{equation}
\end{lemma}
\begin{proof}
Writing $[z^d]$ for coefficient extraction, \eqref{eq:pold} gives
\[
 \frac{p_{m,r}(n;\alpha)}{d!}
 =[z^d](1-z)^{-n-\alpha-1}
 {}_2F_1(m+r+1,r+1-\alpha;2r+2;z).
\]
Euler's transformation \cite[\S15.8]{DLMF} changes the hypergeometric factor into
\[
 (1-z)^{\alpha-m}{}_2F_1(-d,r+1+\alpha;2r+2;z).
\]
For $|\alpha|<r+1$, take $U\sim\operatorname{Beta}(r+1+\alpha,r+1-\alpha)$. The last terminating series is $\E(1-Uz)^d$. Expand
\[
 (1-Uz)^d=\sum_{q=0}^d\binom dq(1-U)^qz^q(1-z)^{d-q}
\]
and use $\E(1-U)^q=(r+1-\alpha)_q/(2r+2)_q$. Extracting the coefficient yields \eqref{eq:pnew}. Polynomial continuation removes the temporary restriction on $\alpha$. In the interval needed for \eqref{eq:pmono}, all summands are nonnegative, the $q=0$ summand is positive, and the only parameter-dependent factor is nonincreasing.
\end{proof}

Define
\begin{equation}\label{eq:Qm}
 Q_m(t)=D_m(t)\bigl(D_m(t)-H_m(t)\bigr).
\end{equation}
For $r\ge1$ and $0\le\alpha\le1$, put
\begin{equation}\label{eq:Sr}
 S_{m,r}(\alpha)=\sum_{n=0}^{\infty}
 \frac{p_{m,r}(n;\alpha)^2}{Q_m(n+\alpha)}.
\end{equation}
The summand is $O(n^{-2r-1})$, and the series converges. Each coefficient of $D_m-H_m$, except the cancelled leading coefficient, is strictly positive. Thus Lemma~\ref{lem:pnew} implies
\begin{equation}\label{eq:Smono}
 S_{m,r}(b)\le S_{m,r}(a)\qquad(0\le a\le b\le1).
\end{equation}
For $0<\alpha<1$, set
\begin{equation}\label{eq:theta}
 T_{m,r}(\alpha)=|\lambda_{m,r}(\alpha)|S_{m,r}(\alpha),\qquad
 \tht=\sum_{\substack{1\le r<m\\r\ \odd}}T_{m,r}(\alpha).
\end{equation}
At $\alpha=0,1$, define $\Theta_{\alpha,m}=0$. These endpoints have no negative coefficients in \eqref{eq:defect}.

\begin{lemma}\label{lem:scalarcriterion}
If $0\le\alpha\le1$ and $\tht\le c$, then
\begin{equation}\label{eq:scalarsuff}
 I-B_{\alpha,m}^*B_{\alpha,m}\ge(1-c)R_{\alpha,m}.
\end{equation}
Moreover,
\begin{equation}\label{eq:halftransfer}
 \Theta_{\alpha,m}\le\frac{28}{27}\Theta_{1/2,m}
 \qquad(1/2\le\alpha\le1).
\end{equation}
\end{lemma}
\begin{proof}
For $r\ge1$, weighted Cauchy--Schwarz gives
\[
 v_{m,r}v_{m,r}^*\le S_{m,r}(\alpha)R_{\alpha,m}.
\]
On $0<\alpha<1$ the negative coefficients are exactly the odd $r$ terms. Discarding all the nonnegative terms in \eqref{eq:defect} proves \eqref{eq:scalarsuff}. The $r=0$ weighted series diverges and is never used in this argument.

For $1/2\le\alpha<1$ and $r\ge1$,
\[
 \frac{|\lambda_{m,r}(\alpha)|}{|\lambda_{m,r}(1/2)|}
 =\frac{\alpha(1-\alpha^2)}{3/8}
 \prod_{j=2}^r\frac{j^2-\alpha^2}{j^2-1/4}
 \le\frac{16}{9\sqrt3}<\frac{28}{27}.
\]
Combine this with \eqref{eq:Smono} and sum over odd $r$. The endpoint $\alpha=1$ is immediate.
\end{proof}

It therefore suffices to prove, for every $m$,
\begin{equation}\label{eq:halftarget}
 \Theta_{\alpha,m}<\frac{15}{16}\qquad(0\le\alpha\le1/2).
\end{equation}
Indeed, $(28/27)(15/16)=35/36$.

\section{Two finite-order scalar estimates}\label{sec:scalar}
The bounds in this section control the actual series \eqref{eq:Sr}; no limit in $m$ is used. Unless stated otherwise, assume
\[
 0<\alpha\le\frac12,\qquad1\le r\le m-2.
\]
Use the notation
\begin{align}
 d&=m-r-1, &k&=2r+1, &N_0&=2m-r-1,\notag\\
 A&=\frac{(m+r)!}{d!}, &s&=r-\alpha,
 &\sigma(\alpha)&=\frac{\sin(\pi\alpha)}{\pi},\label{eq:scalarvars}\\
 E&=d^2+N_0s=m^2-(1+\alpha)N_0,
 &z&=\frac{d^2}{E}, &F&=\frac{A^2}{E^k}.\notag
\end{align}
Also put
\begin{equation}\label{eq:deltaC}
 \delta=1-\frac{kz}{k+1},\qquad
 C=1+\frac{d^2(1+2\alpha)}{2rm^2s}+\frac{kN_0}{2rm^2}.
\end{equation}
All denominators in these formulas are positive.

\subsection{A probability identity and comparison of the full series}
Take independent random variables $U$ and $Q_n$, where
\[
 U\sim\operatorname{Beta}(r+1+\alpha,r+1-\alpha),\qquad
 \mathbb P(Q_n=q)=\frac{\binom dq\binom{n+m}{d-q}}{\binom{n+m+d}{d}}.
\]
The latter is a hypergeometric distribution, with population size $N=n+m+d$, $d$ marked objects, and sample size $d$. Formula \eqref{eq:pnew}, together with
\[
 \binom dq^2q!(n+r+2+q)_{d-q}
 =d!\binom dq\binom{n+m}{d-q},
\]
gives the exact finite identity
\begin{equation}\label{eq:probp}
 p_{m,r}(n;\alpha)=(n+m+1)_d\,\E(1-U)^{Q_n}.
\end{equation}
For $0\le u\le1$, the elementary symmetric-mean inequality implies
\begin{equation}\label{eq:hypergeom}
 \E(1-u)^{Q_n}\le\left(1-\frac{du}{N}\right)^d
 \le e^{-d^2u/N}.
\end{equation}
For completeness, the left side is the $d$th elementary symmetric mean of $N$ numbers, $d$ of which are $1-u$ and the others $1$. With their sum fixed, replacing two of these numbers by their average cannot decrease the elementary symmetric sum: only the term containing their product changes, and its coefficient is nonnegative. Maximization on the compact simplex, followed if necessary by minimizing variance among maximizers, gives the equal-vector bound in \eqref{eq:hypergeom}.

Let
\[
 \phi(u)=\E e^{-uU}={}_1F_1(r+1+\alpha;2r+2;-u),
 \qquad
 c_r(\alpha)=\frac{k\alpha\prod_{j=1}^r(j^2-\alpha^2)}{(k!)^2}.
\]
The beta integral gives the displayed expression for $\phi$ \cite[\S13.4]{DLMF}. Since $|\lambda_{m,r}|=c_r A^2$, \eqref{eq:probp}--\eqref{eq:hypergeom} yield
\[
 v_{m,r}(n;\alpha)
 =\frac{D_m(n)}{D_m(n+\alpha)}
   \frac{\E(1-U)^{Q_n}}{(N+1)_{r+1}}
 \le\frac{\phi(d^2/N)}{(N+1)_{r+1}}.
\]
On the other hand,
\[
 \frac{D_m(n+\alpha)}{H_m(n+\alpha)}
 =\prod_{j=1}^m\left(1+\frac{m}{n+\alpha+j}\right)
 \ge1+\frac{m^2}{n+\alpha+m},
\]
so, using $d\ge1\ge\alpha$,
\begin{equation}\label{eq:diagonalbound}
 \frac1{r_n(\alpha)}\le1+\frac{n+\alpha+m}{m^2}\le1+\frac N{m^2}.
\end{equation}
For every $x\in[N,N+1]$, $x^{r+1}\le(N+1)_{r+1}$ and $\phi(d^2/x)\ge\phi(d^2/N)$. Therefore
\[
 \frac{v_{m,r}(n;\alpha)^2}{r_n(\alpha)}
 \le x^{-2r-2}\phi(d^2/x)^2(1+x/m^2).
\]
Integrate on the unit interval, sum all such intervals, and substitute $u=d^2/x$. This proves
\begin{equation}\label{eq:finiteintegral}
 T_{m,r}(\alpha)\le
 c_r(\alpha)\frac{A^2}{d^{2k}}
 \int_0^{d^2/N_0}u^{k-1}\phi(u)^2
 \left(1+\frac{d^2}{m^2u}\right)\dd u.
\end{equation}
All terms are nonnegative; the comparison includes the entire infinite tail.

\subsection{The Jensen bound}
\begin{lemma}\label{lem:Jensen}
Under the hypotheses of this section,
\begin{equation}\label{eq:Jensen}
 T_{m,r}(\alpha)\le
 \frac{\sigma(\alpha)}s\left(1+\frac1s\right)^{2\alpha}
 F C\delta^{2\alpha}.
\end{equation}
\end{lemma}
\begin{proof}
Put $\eta=2\alpha\in(0,1]$, $b=r+1+\alpha$ and $c=r+1-\alpha=s+1$. The beta integral and $(1-x)^s\le e^{-sx}$ give
\[
 \phi(u)\le\frac{\Gamma(k+1)}{\Gamma(c)}(u+s)^{-b}.
\]
The reflection formula \cite[\S5.5]{DLMF} gives
\[
 c_r(\alpha)\left(\frac{\Gamma(k+1)}{\Gamma(c)}\right)^2
 =k\sigma(\alpha)\frac{\Gamma(b)}{\Gamma(c)}.
\]
Substitute into \eqref{eq:finiteintegral} and then set $t=u/(u+s)$:
\begin{align}\label{eq:incomplete}
 T_{m,r}(\alpha)\le{}&
 k\sigma(\alpha)\frac{\Gamma(b)}{\Gamma(c)}\frac{A^2}{d^{2k}}
 \left[s^{-1-\eta}\Beta_z(k,1+\eta)
 +\frac{d^2}{m^2}s^{-2-\eta}\Beta_z(k-1,2+\eta)\right],
\end{align}
where $\Beta_z(p,q)=\int_0^z t^{p-1}(1-t)^{q-1}\dd t$.

If $V\sim\operatorname{Beta}(k,1)$, concavity gives
\[
 \Beta_z(k,1+\eta)=\frac{z^k}{k}\E(1-zV)^\eta
 \le\frac{z^k}{k}\delta^\eta.
\]
Integration by parts also gives
\[
 \Beta_z(k-1,2+\eta)
 =\frac{(1+\eta)\Beta_z(k,1+\eta)
       +z^{k-1}(1-z)^{1+\eta}}{k-1}.
\]
Use $(1-z)^\eta\le\delta^\eta$ and
\[
 \frac{kd^2(1-z)}{m^2s(k-1)z}=\frac{kN_0}{m^2(k-1)}.
\]
Then \eqref{eq:incomplete} is bounded by
$\sigma(\alpha)\Gamma(b)\Gamma(c)^{-1}s^{-1-\eta}FC\delta^\eta$.
Finally, for a gamma random variable $G$ with shape $c$ and scale $1$,
\[
 \frac{\Gamma(b)}{\Gamma(c)}=\E G^\eta\le(\E G)^\eta=(s+1)^\eta.
\]
This proves \eqref{eq:Jensen}. Both concavity steps have explicitly used $2\alpha\le1$.
\end{proof}

\subsection{A second bound with a summable continuous majorant}
Set
\begin{equation}\label{eq:XiU}
 \Xi_{m,r}=\left[\prod_{j=1}^{2r+1}\left(1+\frac j{m-r-1}\right)\right]^2,
\end{equation}
\begin{equation}\label{eq:Ur}
 \mathcal U_r(\alpha)=\frac{\sigma(\alpha)}{2^{2\alpha}}
 \Beta(r+1/2-\alpha,2\alpha+1)
 \left(1+\frac{2\alpha+1}{2r(r+1)}\right).
\end{equation}
\begin{lemma}\label{lem:continuous}
For $0<\alpha<1$ and $1\le r\le m-2$,
\begin{equation}\label{eq:TU}
 T_{m,r}(\alpha)\le\Xi_{m,r}\mathcal U_r(\alpha).
\end{equation}
Furthermore,
\begin{equation}\label{eq:Usum}
 \sum_{\substack{r\ge1\\r\ \odd}}\mathcal U_r(\alpha)\le\frac5{16}.
\end{equation}
\end{lemma}
\begin{proof}
The proof of \eqref{eq:finiteintegral} also works for $0<\alpha<1$. Enlarge the integral to $(0,\infty)$ and use $d^2/m^2\le1$. The resulting integral is
\[
 \mathcal J_r^+(\alpha)=c_r(\alpha)
 \int_0^\infty u^{2r}\phi(u)^2(1+1/u)\dd u,
\]
multiplied by $\Xi_{m,r}$. For independent $X,Y\sim\operatorname{Beta}(b,c)$,
\[
 \mathcal J_r^+=c_r(\alpha)
 \left[(2r)!\E(X+Y)^{-2r-1}+(2r-1)!\E(X+Y)^{-2r}\right].
\]
In the double beta integral, set $s=x+y$ and $t=x/(x+y)$. Since $c-1=r-\alpha\ge0$ and
$(1-x)(1-y)\le(1-s/2)^2$, enlargement of the $t$ interval to $[0,1]$ gives, for $j=2r+1,2r$,
\[
 \E(X+Y)^{-j}\le
 \frac{\Beta(b,b)2^{2b-j}\Beta(2b-j,2c-1)}{\Beta(b,c)^2}.
\]
Use $c_r(\alpha)(2r)!=\sigma(\alpha)\Beta(b,c)$ and the gamma duplication formula. The first term becomes
$\sigma(\alpha)2^{-2\alpha}\Beta(r+1/2-\alpha,2\alpha+1)$;
the ratio of the second bound to the first is $(2\alpha+1)/(2r(r+1))$. This proves \eqref{eq:TU}.

For \eqref{eq:Usum}, Tonelli's theorem gives
\begin{align*}
 \sum_{r\ \odd}\Beta(r+1/2-\alpha,2\alpha+1)
 &=\int_0^1\frac{t^{1/2-\alpha}(1-t)^{2\alpha-1}}{1+t}\dd t\\
 &\le\frac12\Beta(1-\alpha,2\alpha),
\end{align*}
using $1+t\ge2\sqrt t$. The correction factor in \eqref{eq:Ur} is at most $(2\alpha+5)/4$. Reflection and duplication reduce the resulting bound to
\[
 \frac{(2\alpha+5)\Gamma(\alpha+1/2)}{16\sqrt\pi\Gamma(\alpha+1)}.
\]
The beta integral and H\"older's inequality show that
$\Gamma(\alpha+1/2)/\Gamma(\alpha+1)$ is log-convex on $[0,1]$.
Its endpoint values are $\sqrt\pi$ and $\sqrt\pi/2$, so the ratio is at most $\sqrt\pi\,2^{-\alpha}$. Finally,
$2^\alpha\ge1+\alpha\log2\ge1+2\alpha/5$, which proves \eqref{eq:Usum}.
\end{proof}

\begin{lemma}[The terminal direction]\label{lem:last}
For $m\ge2$ and $0<\alpha<1$,
\begin{equation}\label{eq:last}
 T_{m,m-1}(\alpha)\le\frac{12\alpha}{m4^m}\le\frac{12}{m4^m}.
\end{equation}
\end{lemma}
\begin{proof}
Here $d=0$ and $p_{m,m-1}=1$. One has
\[
 D_m(n+\alpha)\ge D_m(0)\left(1+\frac n{2m}\right)^m,
 \qquad
 |\lambda_{m,m-1}|\le\alpha(2m-1)((m-1)!)^2.
\]
From the first inequality in \eqref{eq:diagonalbound}, valid independently of $d$,
\[
 \frac1{r_n(\alpha)}\le1+\frac2m\left(1+\frac n{2m}\right).
\]
For $p>1$, the sum of $(1+n/(2m))^{-p}$ is at most $1+2m/(p-1)$. The two resulting sums have total coefficient
\[
 1+\frac{2m}{2m-1}+\frac2m\left(1+\frac{2m}{2m-2}\right)<6.
\]
Use $(m-1)!/D_m(0)=1/(m\binom{2m}{m})$ and $\binom{2m}{m}\ge2^m$ to finish. Thus the case $d=0$ never enters formulas with negative powers of $d$.
\end{proof}

\section{An explicit analytic estimate for large orders}\label{sec:large}
\begin{proposition}\label{prop:large}
For every $m\ge32768$ and $0\le\alpha\le1/2$,
\begin{equation}\label{eq:large}
 \Theta_{\alpha,m}<\frac{15}{16}.
\end{equation}
\end{proposition}
\begin{proof}
The case $\alpha=0$ is immediate. Set $M_0=32768$, $q_0=8191/8192$ and $R_0=27$.
We continue to use \eqref{eq:scalarvars}--\eqref{eq:deltaC}.

\smallskip\noindent\emph{Step 1: a factorial-ratio bound.}
Since $\alpha\le1/2$,
\[
 E\ge m^2(1-3/m),\qquad A=m\prod_{j=1}^r(m^2-j^2).
\]
Using $\log(1-x)\le-x$ and $-\log(1-x)\le x/(1-x)$ yields
\[
 \log F\le-\frac{2r^3}{3m^2}+\frac{6r+3}{m(1-3/m)}.
\]
Let $b_0=6/q_0$ and $c_0=3/q_0$. Retain one term $-r^3/(3m^2)$ and maximize the remaining cubic in $r$. Then
\[
 \log F\le-\frac{r^3}{3m^2}
 +\frac{2b_0^{3/2}}{3\sqrt m}+\frac{c_0}m.
\]
Now $b_0<601/100$, $(601/100)^3<225$, $\sqrt{M_0}>180$, and $c_0/M_0<1/8192$. Thus the last two terms are less than $1/18+1/8192$. The inequality $e^x\le(1-x)^{-1}$ and a rational comparison give
\begin{equation}\label{eq:Flarge}
 F\le\frac{53}{50}e^{-r^3/(3m^2)}.
\end{equation}

\smallskip\noindent\emph{Step 2: the first thirteen odd directions.}
For odd $r<R_0$, one has $r\le25$, and
\[
 \log\Xi_{m,r}\le\frac{(2r+1)(2r+2)}{m-r-1}
 \le\frac{2652}{32742}.
\]
Lemmas~\ref{lem:continuous} and $e^x\le(1-x)^{-1}$ imply
\begin{equation}\label{eq:headlarge}
 \sum_{\substack{r<R_0\\r\ \odd}}T_{m,r}(\alpha)
 \le\frac5{16}\frac{32742}{30090}=\frac{321}{944}.
\end{equation}

\smallskip\noindent\emph{Step 3: all higher nonterminal directions.}
Let $\eta=2\alpha\in(0,1]$. If $r\ge27$, then
\[
 \frac1s\le\frac{54}{53r},\qquad
 (1+1/s)^\eta\le\frac{55}{53},\qquad
 C\le1+\frac2{27^2}+\frac3{32768}<\frac{257}{256}.
\]
Furthermore,
\[
 \delta=\frac{N_0s+d^2/(2r+2)}E
 \le\frac1{q_0}\left(\frac{2r}{m}+\frac1{2r}\right).
\]
Combining Lemma~\ref{lem:Jensen}, \eqref{eq:Flarge} and $(x+y)^\eta\le x^\eta+y^\eta$ gives
\begin{equation}\label{eq:tailpoint}
 T_{m,r}(\alpha)\le
 P\frac{\sigma(\alpha)}r e^{-r^3/(3m^2)}
 \left[\left(\frac{2r}m\right)^\eta+(2r)^{-\eta}\right],
\end{equation}
where
\begin{equation}\label{eq:Pconstant}
 P=\frac{53}{50}\frac{54}{53}\frac{55}{53}
   \frac{257}{256}\frac{8192}{8191}
  =\frac{2442528}{2170615}.
\end{equation}
For a positive decreasing function $f$ and odd $R_0$,
\begin{equation}\label{eq:oddgrid}
 \sum_{j=0}^\infty f(R_0+2j)
 \le f(R_0)+\frac12\int_{R_0}^{\infty}f(x)\dd x.
\end{equation}
Apply this first to $x^{-1-\eta}$, discarding the exponential in the second part of \eqref{eq:tailpoint}. Its contribution, before multiplication by $P$, is at most
$\sigma(\alpha)/R_0+1/4$, because $\sigma(\alpha)\le\alpha=\eta/2$.

For the first part use the decreasing function
$x^{\eta-1}e^{-x^3/(3m^2)}$ and
\[
 \int_0^\infty x^{\eta-1}e^{-x^3/(3m^2)}\dd x
 =\frac{(3m^2)^{\eta/3}}3\Gamma(\eta/3).
\]
The first grid term is at most $\sigma(\alpha)/R_0$, since $2R_0\le m$. The integral contribution is at most
\[
 \frac{\sigma(\alpha)}{2\eta}\Gamma(1+\eta/3)
 \left(\frac{24}{m}\right)^{\eta/3}\le\frac14.
\]
Here log-convexity and $\Gamma(1)=\Gamma(2)=1$ give $\Gamma(1+\eta/3)\le1$.
Consequently,
\begin{equation}\label{eq:taillarge}
 \sum_{\substack{27\le r\le m-2\\r\ \odd}}T_{m,r}(\alpha)
 \le P\left(\frac12+\frac2{\pi R_0}\right)
 <P\left(\frac12+\frac2{81}\right)
 =\frac{768944}{1302369}.
\end{equation}
The positive majorants in this step were extended to infinitely many directions; thus no direction tail is omitted.

\smallskip\noindent\emph{Step 4: the endpoint and the final rational comparison.}
Lemma~\ref{lem:last} bounds the terminal direction by $10^{-6}$ for $m\ge M_0$. Equations \eqref{eq:headlarge} and \eqref{eq:taillarge} now give
\begin{equation}\label{eq:largeconstant}
 \Theta_{\alpha,m}\le
 \frac{321}{944}+\frac{768944}{1302369}+\frac1{10^6}
 =\frac{71496550902271}{76839771000000}<\frac{15}{16}.
\end{equation}
The difference from $15/16$ is the positive rational number
$540734410229/76839771000000$. This proves the proposition.
\end{proof}

\section{Finite certificates for the remaining orders}\label{sec:certificates}
This section describes all quantities checked by the accompanying programs. The computations contain no floating-point decisions. Both the analytic meaning of the certificates and the direction of every enclosure are specified below.

\subsection{Orders one through twenty}
Let
\begin{equation}\label{eq:kappa}
 \kappa_{m,r}=(2r+1)\binom{m+r}{2r+1}^{\!2}(r!)^2,
 \qquad K_m=\sum_{\substack{1\le r<m\\r\ \odd}}\kappa_{m,r}S_{m,r}(0).
\end{equation}
For $0\le\alpha\le1$, \eqref{eq:Smono} and
\[
 |\lambda_{m,r}(\alpha)|
 =\alpha\kappa_{m,r}\prod_{j=1}^r(1-\alpha^2/j^2)
 \le\alpha(1-\alpha^2)\kappa_{m,r}
\]
imply
\begin{equation}\label{eq:Kbound}
 \Theta_{\alpha,m}\le\alpha(1-\alpha^2)K_m.
\end{equation}

For each odd $r<m\le20$, the supplement supplies positive rational numbers $b_{m,r}$ and $C_{m,r}^{\rm pol}$. The polynomial certificate is that every coefficient of
\begin{equation}\label{eq:polycert}
 C_{m,r}^{\rm pol}Q_m(t)
 -\kappa_{m,r}p_{m,r}(t;0)^2(t+b_{m,r})^{2r+1}
\end{equation}
is nonnegative. Both polynomials have degree $2m-1$. Therefore
\begin{equation}\label{eq:polytail}
 \kappa_{m,r}S_{m,r}(0)
 \le C_{m,r}^{\rm pol}
 \left(b_{m,r}^{-2r-1}+\frac{b_{m,r}^{-2r}}{2r}\right).
\end{equation}
The last step is the integral comparison for the positive decreasing function $(t+b_{m,r})^{-2r-1}$; it includes all $n\ge0$.

\begin{lemma}[Low-order certificate]\label{lem:lowcert}
The supplied rational data satisfy \eqref{eq:polycert} for all odd $r<m$, $2\le m\le20$, and
\[
 \sum_{\substack{r<m\\r\ \odd}}C_{m,r}^{\rm pol}
 \left(b_{m,r}^{-2r-1}+\frac{b_{m,r}^{-2r}}{2r}\right)<\frac34.
\]
Consequently $\Theta_{\alpha,m}<3/10$ for $2\le m\le20$ and $0\le\alpha\le1$. For $m=1$, $\Theta_{\alpha,1}=0$.
\end{lemma}
\begin{proof}[Certificate verification]
The file \texttt{low\_order/certificate\_data.json} contains all the rational pairs.
The standard-library verifier \texttt{low\_order/verify\_certificates.py} reconstructs $p_{m,r}(t;0)$ from \eqref{eq:pold}, constructs $Q_m$, verifies all $2770$ coefficient comparisons in \eqref{eq:polycert}, and checks the indicated aggregate bound for each of the nineteen orders. It verifies the complete order and direction lists before testing the inequalities. All checks pass with exact rational arithmetic. Finally,
$\max_{[0,1]}\alpha(1-\alpha^2)=2/(3\sqrt3)<2/5$,
so \eqref{eq:Kbound} gives the claimed $3/10$ bound.
\end{proof}

\subsection{A H\"older envelope for the continuous majorant}
For $r\ge1$, let
\[
 W_r(\alpha)=2^{-2\alpha}\Beta(r+1/2-\alpha,2\alpha+1).
\]
Its beta integral has the form of an integral of a positive function raised to a power affine in $\alpha$; H\"older's inequality makes $W_r$ log-convex on $[0,1/2]$. Since
\[
 W_r(0)=\frac1{r+1/2},\qquad W_r(1/2)=\frac1{2r(r+1)},
\]
one obtains
\begin{equation}\label{eq:holderenvelope}
 W_r(\alpha)\le\frac1{r+1/2}
 \left(\frac{r+1/2}{2r(r+1)}\right)^{2\alpha}.
\end{equation}
This bound retains the parameter exponent, instead of replacing all parameters by zero.

\subsection{An envelope on a complete order--parameter rectangle}
Let $L\le m\le U$ be an interval of integers with $L\ge21$, and let
$\ell\le\alpha\le u\le1/2$. For each $1\le r\le U-2$, put
\begin{equation}\label{eq:rectvars}
 \mu=\max\{L,r+2\},\qquad q=r+1,\qquad k=2r+1,
 \qquad N=2\mu-q,\qquad h=N/\mu^2.
\end{equation}
Every actual nonterminal direction has $m\ge\mu$. Define the following positive rational quantities:
\begin{align}
 P_{U,r}&=\prod_{j=1}^r\left(1-\frac{j^2}{U^2}\right)^2,
 &\overline F&=\frac{P_{U,r}}{[1-(1+u)h]^k},\label{eq:Fbar}\\
 \overline\delta&=1-\frac{k}{k+1}
 \frac{(\mu-q)^2}{\mu^2-(1+\ell)N},
 &X&=\left(1+\frac1{r-u}\right)\overline\delta,\label{eq:Xbar}\\
 \overline C&=1+\frac{(U-q)^2}{U^2}\frac{1+2u}{2r(r-u)}
 +\frac{kh}{2r}.&&\label{eq:Cbar}
\end{align}
For nonintegral powers set
\begin{equation}\label{eq:Zbar}
 Z=\begin{cases}X^{2\ell},&X\le1,\\ X^{2u},&X>1.\end{cases}
\end{equation}
The sine product \cite[\S4.22]{DLMF} gives the rational upper bound
\begin{equation}\label{eq:sigmabar}
 \sigma(\alpha)\le\overline\sigma(u)
 :=u\prod_{j=1}^{16}\left(1-\frac{u^2}{j^2}\right).
\end{equation}
Indeed, $\sigma$ is increasing on $[0,1/2]$, and every omitted factor of its infinite product is in $(0,1]$.

\begin{lemma}\label{lem:rectangle}
Every actual nonterminal direction in the above rectangle satisfies
\begin{equation}\label{eq:Jbar}
 T_{m,r}(\alpha)\le
 J_{L,U,r,\ell,u}:=
 \frac{\overline\sigma(u)}{r-u}\,\overline F\,\overline C\,Z.
\end{equation}
For $r\le9$, it also satisfies
\begin{equation}\label{eq:Hbar}
 T_{m,r}(\alpha)\le H_{L,U,r,\ell,u}:=
 \Xi_{\mu,r}\frac{\overline\sigma(u)}{r+1/2}
 \left(\frac{r+1/2}{2r(r+1)}\right)^{2\ell}
 \left(1+\frac{2u+1}{2r(r+1)}\right).
\end{equation}
\end{lemma}
\begin{proof}
The needed monotonicities are factorwise; no monotonicity of the full quantity $\Theta$ is presumed. The function
$h(m)=(2m-q)/m^2$ decreases for $m>q$, whereas
$\prod_{j=1}^r(1-j^2/m^2)^2$ increases. Thus $F\le\overline F$.
For $E=m^2-(1+\alpha)(2m-q)$ and $z=(m-q)^2/E$,
\[
 \frac{\partial z}{\partial m}
 =\frac{2(m-q)(r-\alpha)m}{E^2}>0,
 \qquad
 \frac{\partial z}{\partial\alpha}
 =\frac{(m-q)^2(2m-q)}{E^2}>0.
\]
It follows that $\delta\le\overline\delta$. Also $s^{-1}\le(r-u)^{-1}$. The first nonconstant term of $C$ is bounded using $m=U$, $\alpha=u$; the second uses $m=\mu$. This gives $C\le\overline C$.
The quantity $[(1+1/s)\delta]^{2\alpha}$ is at most $Z$ by the two branches in \eqref{eq:Zbar}. Lemma~\ref{lem:Jensen} proves \eqref{eq:Jbar}.

The product defining $\Xi_{m,r}$ decreases in $m$, and the base in \eqref{eq:holderenvelope} is below $1$. Hence Lemma~\ref{lem:continuous} and \eqref{eq:holderenvelope} give \eqref{eq:Hbar}. We use this second bound only for $r\le9$ to keep its evaluation inexpensive; its validity is not restricted to those directions.
\end{proof}

Combining the direction bounds gives a single majorant for the whole rectangle:
\begin{equation}\label{eq:rectangleB}
 \mathcal B(L,U;\ell,u)=\frac{12}{L4^L}+
 \sum_{\substack{1\le r\le U-2\\r\ \odd}}
 \begin{cases}
 \min\{J_{L,U,r,\ell,u},H_{L,U,r,\ell,u}\},&r\le9,\\
 J_{L,U,r,\ell,u},&r>9.
 \end{cases}
\end{equation}
Directions absent at a particular $m$ only add nonnegative terms to this majorant. The actual terminal direction $r=m-1$ is covered separately by $12/(L4^L)$ using Lemma~\ref{lem:last}. Thus $\Theta_{\alpha,m}\le\mathcal B$ throughout the rectangle, even if $r=m-1$ appears elsewhere as an additional positive envelope term.

\subsection{The complete finite verification}
Generate consecutive integer blocks by
\begin{equation}\label{eq:blockrule}
 L_0=21,\qquad U_j=\min\{32767,\lfloor11L_j/10\rfloor\},
 \qquad L_{j+1}=U_j+1.
\end{equation}
This gives $75$ blocks, with no missing or repeated orders. Use the sixteen closed parameter intervals
\begin{equation}\label{eq:parametergrid}
 [\ell,u]=[j/32,(j+1)/32],\qquad j=0,\ldots,15.
\end{equation}

\begin{lemma}[Intermediate-order certificate]\label{lem:middlecert}
For every one of the $1200$ rectangles specified by \eqref{eq:blockrule}--\eqref{eq:parametergrid},
\begin{equation}\label{eq:middlecert}
 \mathcal B(L,U;\ell,u)<\frac9{10}.
\end{equation}
Consequently $\Theta_{\alpha,m}<9/10$ whenever $21\le m\le32767$ and $0\le\alpha\le1/2$.
\end{lemma}
\begin{proof}[Certificate verification]
The program \texttt{verify\_closure.py} recomputes \eqref{eq:rectangleB} from the deterministic rules above. It stores nonnegative upper bounds on the grid $2^{-96}\mathbb Z$. Every addition, multiplication, division and positive integer power is rounded upward. The exponents in \eqref{eq:Zbar} and \eqref{eq:Hbar} are integer multiples of $1/16$, and are enclosed by four upward-rounded integer square roots followed by an integer power. The sine factor is enclosed by the rational product \eqref{eq:sigmabar}. Thus every output is a rigorous rational upper bound; implementation details are given in Appendix~\ref{app:verification}.

All $1200$ rectangles pass the integer test $10N<9\cdot2^{96}$, where $N/2^{96}$ is the computed total upper bound. The largest certified value is
\begin{equation}\label{eq:worstcert}
 \frac{69781001268962619723556217951}
 {79228162514264337593543950336}
 <0.880760061253<\frac9{10}.
\end{equation}
The displayed decimal is rounded upward and is not used in the test.
The computation evaluates $2{,}956{,}912$ direction--parameter cell terms.
Since Lemma~\ref{lem:rectangle} bounds complete rectangles and \eqref{eq:finiteintegral} already bounds the full $n$ series, neither parameter sampling nor a truncated operator is involved.
\end{proof}

\section{Completion of the proof and the negative parameters}\label{sec:completion}
\begin{proof}[Proof of Theorem~\ref{thm:main}]
Boundedness was established in Section~\ref{sec:defect}. On the half interval $0\le\alpha\le1/2$, Lemmas~\ref{lem:lowcert} and \ref{lem:middlecert}, and Proposition~\ref{prop:large}, give the following complete coverage:
\begin{center}
\begin{tabular}{@{}lll@{}}
\toprule
Order range & Upper bound for $\Theta_{\alpha,m}$ & Justification\\
\midrule
$1\le m\le20$ & $3/10$ & Rational coefficient certificates\\
$21\le m\le32767$ & $9/10$ & Rational rectangle certificates\\
$m\ge32768$ & $15/16$ & Analytic estimate\\
\bottomrule
\end{tabular}
\end{center}
Thus \eqref{eq:halftarget} holds for all positive integer orders. By \eqref{eq:halftransfer}, $\Theta_{\alpha,m}<35/36$ on all of $[0,1]$.
Lemma~\ref{lem:scalarcriterion} and \eqref{eq:transfer} prove \eqref{eq:mainweight} on this basic interval.

Let $S$ be the unilateral shift and $P_0=e_0e_0^*$. Comparison of matrix entries gives
\begin{equation}\label{eq:shift}
 M_{\alpha,m}S=SM_{\alpha+1,m},\qquad
 S^*R_{\alpha,m}S=R_{\alpha+1,m}.
\end{equation}
For $T=M_{\alpha,m}$ and $A=S^*TS=M_{\alpha+1,m}$, the invariant-subspace identity is
\begin{equation}\label{eq:shiftdefect}
 A^*A-AA^*=S^*(T^*T-TT^*)S+S^*TP_0T^*S.
\end{equation}
If $T^*T-TT^*\ge cT^*R_{\alpha,m}T$, then \eqref{eq:shift}--\eqref{eq:shiftdefect} give
$A^*A-AA^*\ge cA^*R_{\alpha+1,m}A$.
Repeatedly applying this implication proves \eqref{eq:mainweight} for every $\alpha\ge0$.

It remains to exclude $-1<\alpha<0$. Fix such a parameter and any $m\ge1$.
The monic degrees of the polynomials give, as $n\to\infty$,
\[
 r_n(\alpha)=\frac{m^2}{n}+O(n^{-2}),\qquad
 v_{m,0}(n)=\frac1n+O(n^{-2}),\qquad
 v_{m,r}(n)=O(n^{-r-1})\quad(r\ge1).
\]
Also $\lambda_{m,0}(\alpha)=m^2\alpha$.
Take $x^{(N)}=\sum_{n=0}^Ne_n$. The exact defect identity \eqref{eq:defect} yields
\begin{equation}\label{eq:negative}
 \|x^{(N)}\|^2-\|Bx^{(N)}\|^2
 =m^2\alpha(\log N)^2+O(\log N)<0
\end{equation}
for all sufficiently large $N$. Indeed, only the $r=0$ rank-one term has a squared logarithmic partial sum; the remaining rank-one partial sums are bounded.
Now $y^{(N)}=Lx^{(N)}$ is finitely supported, and
\[
 My^{(N)}=x^{(N)},\qquad M^*y^{(N)}=Bx^{(N)}.
\]
Equation \eqref{eq:negative} is therefore the quadratic form of $M^*M-MM^*$ at $y^{(N)}$, which is negative. Thus $M_{\alpha,m}$ is not hyponormal. This proves both directions of the theorem.
\end{proof}

\begin{remark}
The special case $m=2$ gives an affirmative answer to \cite[Conjecture 3.2]{R18}, and $m=3$ includes both noninteger parameter intervals omitted from \cite[Corollary 2.4]{R19}. The proof treats the arbitrary-order algebra independently of the finite computation. In particular, checking finitely many orders is not used as a substitute for the defect formula at all orders.
\end{remark}

\appendix
\section{Verification specification and reproducibility}\label{app:verification}
\subsection{Exact directed rounding}
Let $Q=2^{96}$. An integer $x\ge0$ represents an upper bound $x/Q$ for a nonnegative real quantity. For nonnegative integers $p$ and positive integers $q$, the verifier uses
\[
 \operatorname{rat}(p,q)=\left\lceil\frac{Qp}{q}\right\rceil,
 \qquad
 \operatorname{mul}(x,y)=\left\lceil\frac{xy}{Q}\right\rceil.
\]
Positive integer powers are evaluated by repeated squaring with this multiplication. The upward square-root operation returns the smallest integer $z$ satisfying $z^2\ge xQ$. Four iterations enclose the sixteenth root. These operations preserve upper bounds by induction. The essential routines are:
\begin{lstlisting}[language=Python]
from math import isqrt
BITS = 96
ONE = 1 << BITS

def ceildiv(p, q):
    assert p >= 0 and q > 0
    return (p + q - 1) // q

def rat(p, q=1):
    return ceildiv(p * ONE, q)

def mul(x, y):
    assert x >= 0 and y >= 0
    return ceildiv(x * y, ONE)

def power(x, n):
    y = ONE
    while n:
        if n & 1:
            y = mul(y, x)
        n >>= 1
        if n:
            x = mul(x, x)
    return y

def root16(x):
    for _ in range(4):
        n = x * ONE
        z = isqrt(n)
        x = z + (z * z < n)
    return x

def pow16(x, p):
    if p == 0:
        return ONE
    if p == 16:
        return x
    return power(root16(x), p)
\end{lstlisting}
The distributed verifier rejects Python's \texttt{-O} and \texttt{-OO} modes before executing the checks, so that assertion removal cannot lead to a false success report.

All rational denominators in \eqref{eq:Fbar}--\eqref{eq:Cbar} are evaluated exactly and proved positive by $\mu\ge r+2$, $r\ge1$, $u\le1/2$. In particular, the algorithm never divides by a rounded-up denominator when an upper bound is required. If an upper enclosure of $X$ crosses $1$, the branch in \eqref{eq:Zbar} remains safe: below $1$ the smaller exponent is used, and above $1$ the larger exponent yields a bound at least $1$ whenever the exact base was below $1$.

For each order block, the product $P_{U,r}$ is accumulated successively in $r$, including even $r$ in the product although only odd $r$ are added to the total. If a very small positive product rounds up to $2^{-96}$, it remains a valid upper bound; it is never replaced by zero. The terminal-direction term is added separately in every parameter interval.

\subsection{The finite data and their tests}
The low-order data contain one pair $(b_{m,r},C_{m,r}^{\rm pol})$ for each odd $r<m$, $2\le m\le20$. The verifier checks that this list is complete, reconstructs the polynomials from \eqref{eq:pold}, verifies every coefficient of \eqref{eq:polycert}, and then proves each aggregate bound in Lemma~\ref{lem:lowcert}. The supplied numbers are witnesses to finite rational inequalities, rather than values trusted without verification.

The intermediate-order verifier has no fitted constants or floating-point input. It generates \eqref{eq:blockrule}--\eqref{eq:parametergrid}, evaluates the rational envelopes of Lemma~\ref{lem:rectangle}, encloses the fractional powers as above, and tests \eqref{eq:middlecert}. The emitted certificate records all $1200$ dyadic upper bounds. The wrapper verifies complete order and parameter coverage and compares the recomputed numerators to the archived certificate.

The principal files, relative to the \texttt{supplementary} directory, are:
\begin{center}
\begin{tabular}{@{}ll@{}}
\toprule
File & Role\\
\midrule
\texttt{verify\_all.py} & One-command verification and coverage checks\\
\texttt{verify\_closure.py} & Recomputes every rectangle bound\\
\texttt{closure\_certificate.json} & Archived dyadic upper bounds\\
\texttt{low\_order/verify\_certificates.py} & Low-order rational polynomial checks\\
\texttt{low\_order/certificate\_data.json} & Low-order rational witnesses\\
\texttt{audit\_fresh.py} & Additional direct finite-column checks\\
\bottomrule
\end{tabular}
\end{center}
Run the essential verification from the supplement directory with
\begin{lstlisting}
python verify_all.py
\end{lstlisting}
The essential checks use only the Python standard library. The optional symbolic audit \texttt{audit\_general\_formula.py} requires SymPy and is not part of the finite proof certificate. Additional finite-column and arithmetic audits are diagnostic checks; the arbitrary-order algebra is proved in Section~\ref{sec:defect}, not inferred from those examples.

The electronic supplement includes the verification logs and a SHA-256 manifest. This is an integer-arithmetic computer-assisted proof, not a verification in a formal proof assistant. The executable checks certify finite inequalities; the operator-theoretic reductions and analytic estimates must be read together with them.

\end{document}